\documentclass[a4paper,12pt]{article}%
\usepackage[T1]{fontenc}
\usepackage{amsfonts}
\usepackage{amssymb}
\usepackage{textcomp}
\usepackage{typearea}
\usepackage{enumerate}
\usepackage{amsmath}
\usepackage{amsthm}
\usepackage{a4}
\usepackage{vmargin}
\usepackage{mathrsfs}
\usepackage[center]{subfigure}
\usepackage{todonotes}
\usepackage{scrpage2}
\usepackage[font=small,format=plain,labelfont=bf,up,textfont=it,up]{caption}
\usepackage{enumitem}
\usepackage{algorithm}
\usepackage{algorithmic}
\usepackage{graphicx}%
\usepackage{url}
\providecommand{\U}[1]{\protect\rule{.1in}{.1in}}
\newtheorem{theorem}{Theorem}

\newtheorem{lemma}[theorem]{Lemma}

\newtheorem{proposition}[theorem]{Proposition}
\newtheorem{remark}[theorem]{Remark}

\graphicspath{{Pictures/}}
\newcommand{\N}{\mathbb{N}}

\newcommand{\R}{\mathbb{R}}
\newcommand{\Co}{\mathbb{C}}

\newcommand{\vertiii}[1]{{\left\vert\kern-0.25ex\left\vert\kern-0.25ex\left\vert#1
\right\vert\kern-0.25ex\right\vert\kern-0.25ex\right\vert}}

\numberwithin{equation}{section}
\numberwithin{theorem}{section}

\numberwithin{figure}{section}
\newcommand{\beq}{\begin{equation}}
\newcommand{\eeq}{\end{equation}}
\begin{document}

\title{Efficient computation of highly oscillatory integrals  by using
QTT tensor approximation}
\author{Boris Khoromskij\thanks{Max-Planck-Institut f\"ur Mathematik in den
Naturwissenschaften, Inselstr. 22-26, 04103 Leipzig, Germany, e-mail:
\texttt{bokh@mis.mpg.de}}\hspace{2cm}
Alexander Veit\thanks{Department of Computer Science, University of Chicago, 1100 East 58th Street, Chicago, IL 60637, e-mail:
\texttt{aveit@uchicago.edu}} }
\date{}
\maketitle

\begin{abstract}
We propose a new method for the efficient approximation of a class of 
highly oscillatory weighted integrals where the oscillatory function 
depends on the frequency parameter $\omega \geq 0$, typically varying in a large interval.
Our approach is based, for fixed but arbitrary oscillator, 
on the pre-computation and low-parametric approximation of certain $\omega$-dependent 
prototype functions whose evaluation leads in a straightforward way to recover
the target integral. 
The difficulty that arises is that these prototype 
functions consist of oscillatory integrals which makes them difficult to evaluate. Furthermore they have to be approximated typically in large intervals. 
Here we use the quantized-tensor train (QTT) approximation method for functional $m$-vectors
of logarithmic complexity in $m$
in combination with a cross-approximation scheme for TT tensors.
This allows the accurate approximation and efficient storage of these functions in 
the wide range of grid and frequency parameters. 
Numerical examples illustrate the efficiency of the QTT-based numerical integration
scheme on various examples in one and several spatial dimensions.

\medskip

\textbf{AMS subject classifications: } 65F30, 65F50, 65N35, 65D30
\medskip

\textbf{Keywords: } highly oscillatory integrals, quadrature, tensor representation, 
QTT tensor approximation.

\end{abstract}

\section{Introduction and Problem Setting}

In this paper we are interested in the efficient approximation of (highly) oscillatory integrals. 
In the most general setting these integrals are of the form
\beq
\label{genInt}
\int_\Omega f(x) h_\omega (x) dx,
\eeq
where $\Omega\subset\mathbb{R}^d$, $d\in\N$, is a general open domain, 
$h_\omega$ is an oscillatory function where the 
parameter $\omega\geq 0$ determines the rate of oscillation and $f$ is a 
non-oscillatory (typically analytic) function. 
An important special case occurs if the oscillatory function is the imaginary exponential 
function with oscillator $g$, i.e.,
\beq
\label{imagexp}
h_\omega (x) = \operatorname{e}^{\operatorname{i}\omega g(x)}.
\eeq
This type of oscillatory integrals  has been in the main focus of research in recent 
years since they play an important role in 
a wide range of applications. Prominent examples include the solution of highly 
oscillatory differential equations via the modified 
Magnus expansion (see \cite{Iserles_Magnus,HocL03}), boundary integral formulations 
of the Helmholtz equation \cite{KhSaVe}, the evaluation of special 
functions and orthogonal expansions (e.g. Fourier series, modified Fourier series) 
(see \cite{IsNo}), lattice summation techniques and ODEs/PDEs with oscillating and quasi-periodic coefficients 
\cite{VeBoKh:Ewald:14,vekh-lattice2-2014,khor-survey-2014}. 

Other types of oscillatory functions that can be found in the literature include 
the Bessel oscillator 
$h_\omega (x) = J_\nu(\omega x)$ (see \cite{Xiang})  and functions of the form 
$h_\omega (x) = v(\sin(\omega \theta(x)) )$ (see \cite{IserlesLevin}),
as well as some examples considered in \cite{Trefethen}.

For $d=1$ an obvious way to obtain an approximation of \eqref{genInt} is Gaussian quadrature. 
For large $\omega$ however 
such standard approaches become ineffective since the number of quadrature points has 
to be chosen proportional to $\omega$ 
in order to resolve the oscillations of the integrand. Therefore several alternative 
approaches have been developed to  
overcome this difficulty. The most successful methods include the asymptotic expansion, 
Filon-type methods, Levin-type methods 
and numerical steepest descent (see e.g. \cite{Olver,Iserles2005,Iserles2006, Olver2, Huybrechs, Olver2006}) and 
recently introduced Gaussian quadrature rules with complex weight functions 
(see \cite{complexGauss1, complexGauss2}). 
Although these methods are mathematically elegant they can typically not be applied in 
a ``black-box'' fashion since either 
derivatives of $f$ and $g$ are involved, moments 
$\int_a^b x^k \operatorname{e}^{\operatorname{i}\omega g(x)}$ must be known or 
computations in the complex plane have to be performed. Furthermore these methods get more 
complicated (or even non-applicable) 
if the oscillator $g$ has stationary points (i.e. points where $g^\prime (x)$ vanishes), 
multidimensional integrals are considered 
or $f$ and $g$ are not analytic. For general oscillators $h_\omega(x)$ it is typically 
not known how these methods can be applied (see however \cite{Xiang, IserlesLevin}).

In this paper we propose a Filon-like method which is based, for fixed but arbitrary 
oscillators $g$ or $h_\omega$ 
respectively, on the pre-computation and approximation of certain $\omega$-dependent 
prototype functions whose evaluation 
leads in a straightforward way to approximations of \eqref{genInt}. The difficulty 
that arises is that these prototype 
functions consist of oscillatory integrals which makes 
them difficult to evaluate. Furthermore they have to be approximated typically in large intervals. 
Here we use the quantized-tensor train (QTT) approximation method for functional $M$-vectors
\cite{khor-qtt-2009pre,khor-qtt-2011} of logarithmic complexity in $M$
in combination with a cross-approximation scheme for TT tensors introduced in \cite{Osel10}. 
This allows the accurate approximation and efficient storage of these functions in 
the wide range of grid and frequency parameters. 
Literature surveys on tensor methods can be found in 
\cite{KoldaB:07,khor-survey-2011,hackbusch-2012,larskres-survey-2013,khor-survey-2014,dc-phd}.

The QTT approximation applies to the quantized image of the target discrete function,
obtained by its isometric folding transformation to
the higher dimensional {\it quantized tensor space}. 
For example, a vector of size $M=2^L$  can be successively 
reshaped by a diadic folding to an $L$-fold tensor in $\bigotimes_{j=1}^L \mathbb{R}^{2}$
of the irreducible mode size $m=2$ (quantum of information), then
the  low-rank approximation in the canonical or TT  format can be applied consequently.
The rigorous justification of the QTT approximation method for rather general classes of functional vectors
was first presented in \cite{khor-qtt-2009pre}. For our particular application in this paper the most
important result is the existance of rank-$1$ QTT representation of the complex 
exponential $M$-vector $\{e^{i \omega n}\}_{n=0}^{M-1}$.
The QTT-type representation for $2^L \times 2^L$ matrices was introduced in
\cite{osel-2d2d-2010}, see survey papers \cite{khor-survey-2011,khor-survey-2014} for 
further references on the topic.

The quadrature scheme proposed in this article does not require analytic knowledge 
about $f,g$ or $h_\omega$ respectively and can therefore be 
applied in a black-box fashion. One condition for its efficiency  is that $f$ can be well 
approximated by polynomials 
which is typically the case for analytic, non-oscillatory functions. The quadrature error of the
 method can be easily 
estimated and controlled in terms of $f$. Furthermore the method is uniformly accurate for all 
considered $\omega$.\vspace{\baselineskip}

Since it is the most important case in practice we carry out the description and the analysis of our 
scheme only for the case \eqref{imagexp}. 
We emphasize, however, that the method can be easily applied also to the more general situation 
\eqref{genInt} (see Section \ref{Sec:NumExp}). We introduce the notation
\beq
\label{specInt}
I(\omega,f):=\int_\Omega f(x) \operatorname{e}^{\operatorname{i}\omega g(x)} dx.
\eeq
We consider the oscillator $g$ and the domain $\Omega$ to be arbitrary but fixed and are interested 
in the efficient computation 
of \eqref{specInt} for different real-valued functions $f$ and different values of $\omega$. 
Without loss of generality we assume from now on that $|g(x)|\leq 1$ for all $x\in\Omega$. 
A separation of the real- and imaginary part of $I(\omega,f)$ leads to the integrals
\[
I_{\mathscr{R}}(\omega,f):=\int_\Omega f(x) \cos(\omega g(x)) dx,
\quad I_{\mathscr{I}}(\omega,f):=\int_\Omega f(x) \sin(\omega g(x)) dx
\]
which will be considered in the following.
\vspace{\baselineskip}

The remainder of the paper is structured as follows.
Section \ref{sec:QTT_Approx} recalls the main ideas of the QTT approximation of functional vectors.
The central Section \ref{sec:ApproxProc} presents the basic QTT approximation scheme for the fast computation
of one- and multidimensional oscillating integrals. 
Section \ref{sec:QTT_Represent_I} describes and theoretically analyzes the QTT tensor 
approximation of special functions of interest, while Section \ref{Sec:NumExp} presents the numerical illustrations.

\section{Quantized-TT approximation of functional vectors}\label{sec:QTT_Approx}

A real tensor of order $d$ is defined as an element of finite dimensional Hilbert space
$\mathbb{W}_{\bf m} = \bigotimes_{\ell=1}^d X_\ell$
of the $d$-fold, $M_1\times ... \times M_d$ real-valued arrays, where
$X_\ell=\mathbb{R}^{M_\ell}$ and ${\bf m}=(M_1,\ldots ,M_d)$.
A tensor $\textsf{\textbf{A}}\in \mathbb{R}^{\cal I}$ with ${\cal I}=I_1\times ... \times I_d$,
can be represented entrywise by
\[
\textsf{\textbf{A}}=[\textsf{A}(i_1,...,i_d)]\equiv [\textsf{A}_{i_1,...,i_d}]
\quad \mbox{with}\quad
i_\ell\in I_\ell:=\{1,...,M_\ell\}. 
\]
The Euclidean scalar product, $\left\langle \cdot ,\cdot \right\rangle :
\mathbb{W}_{\bf m}\times \mathbb{W}_{\bf m}\to \mathbb{R}$,
is defined by
$$
\left\langle \textsf{\textbf{A}} , \textsf{\textbf{B}} \right\rangle:=
\sum_{{\bf i}\in {\cal I}} \textsf{A}({\bf i}) \textsf{B}({\bf i}), \quad
\textsf{\textbf{A}},\textsf{\textbf{B}}\in \mathbb{W}_{\bf m}.
$$ 
The storage size for a $d$th order tensor scales exponentially in $d$, 
$\operatorname{dim}(\mathbb{W}_{{\bf m}})=M_1\cdots M_d$ (the so-called "curse of dimensionality").
For ease of presentation we further assume that $M_\ell=M$ for $\ell=1,...,d$.

The efficient low-parametric representations of $d$th order tensors can be realized by using
low-rank separable decompositions (formats). 
The commonly used canonical and Tucker tensor formats \cite{KoldaB:07} 
are constructed by combination of
the simplest separable elements given by rank-$1$ tensors,
\[
 \textsf{\textbf{A}}= \bigotimes_{\ell=1}^d \textsf{A}^{(\ell)},
\quad \textsf{A}^{(\ell)}\in \mathbb{R}^{M},
\]
which can be stored with $d M$ numbers.

In this paper we apply the factorized representation of $d$th order tensors in the 
tensor train (TT) format \cite{ot-tt-2009}, which is the particular case of 
the so called matrix product states (MPS) decomposition. The latter was introduced
since longer in the physics community and successfully applied in quantum 
chemistry computations and in spin systems modeling 
\cite{white-dmrg-1993,Vidal-Effic-Simul-Quant-comput-2003,verstraete-dmrg-2004}.

For a given rank parameter ${\bf r}=(r_0,...,r_d)$, and the respective
index sets $J_\ell=\{1,...,r_\ell\}$ ($\ell=0,1,...,d$),
with the constraint $ J_0= J_d=\{1\}$ (i.e., $r_0=r_d=1$),
the rank-${\bf r}$ TT format contains all elements
$\textsf{\textbf{A}}=[\textsf{A}(i_1,...,i_d)]\in\mathbb{W}_{\bf m} $
which can be represented as the contracted products of $3$-tensors 
over the $d$-fold product index set ${\cal J}:=\times_{\ell=1}^d J_\ell$, such that
\[
\textsf{\textbf{A}}=\sum\limits_{{\bf \alpha}\in {\cal J}} \textsf{A}^{(1)}_{1,\alpha_1}
\otimes \textsf{A}^{(2)}_{\alpha_1,\alpha_2} \otimes 
\cdots \otimes \textsf{A}^{(d)}_{\alpha_{d-1},1},
\]
where 
$\textsf{A}^{(\ell)}_{\alpha_{\ell-1},\alpha_{\ell}}   \in \mathbb{R}^{M}$, 
($\ell=1,...,d$), and 
$\textsf{A}^{(\ell)}=[\textsf{A}^{(\ell)}_{\alpha_{\ell-1},\alpha_{\ell}}]$
is the vector-valued $r_{\ell-1}\times r_{\ell}$ matrix ($3$-tensor). 
The TT representation reduces the storage cost to $O(d r^2 M)$, $r=\max {r_\ell}$.

In the case of large mode size, the asymptotic storage for a $d$th order tensor 
can be reduced to logarithmic scale $O(d \log M)$ by using quantics-TT (QTT) 
tensor approximation \cite{khor-qtt-2009pre,khor-qtt-2011}. 
In our paper we apply this 
approximation techniques to long $M$-vectors generated by sampling certain highly-oscillating functions
on the uniform grid.\vspace{\baselineskip}

The QTT-type approximation of an $M$-vector with $M=q^L$, $L\in \mathbb{N}$, $q=2,3,...$,
is defined as the 
tensor decomposition (approximation) in the canonical, TT or some related format applied 
to a tensor 
obtained by the folding (reshaping) of the initial  long vector to an $L$-dimensional 
$q\times \ldots \times q$ data array that is thought as an element of the quantized tensor 
space $\mathbb{Q}_{{q},L}= \bigotimes_{j=1}^{L}\mathbb{K}^{q},
 \; \mathbb{K}\in \{\mathbb{R},\mathbb{C}\}$. A vector 
$
X=[X(i)]_{i\in I}\in \mathbb{W}_{{M}},
$
is reshaped to its quantics image in $\mathbb{Q}_{q,L}$ by $q$-adic folding, 
\[
\mathcal{F}_{q,L}: X \to \textsf{\textbf{Y}}
=[Y({\bf j})]\in \mathbb{Q}_{q,L}, \quad {\bf j}=\{j_{1},\ldots,j_{L}\}, 
\]
with  $j_{\nu}\in \{1,2\}$ for $\nu=1,...,L$,
where for fixed $i$, we have $Y({\bf j}):= X(i)$, and $j_\nu=j_\nu(i)$ is defined via $q$-coding,
$
j_\nu - 1= C_{-1+\nu}, 
$
such that the coefficients $C_{-1+\nu} $ are found from the
$q$-adic representation  of $i-1$ (binary coding for $q=2$),
\[
i-1 =  C_{0} +C_{1} q^{1} + \cdots + C_{L-1} q^{L-1}\equiv
\sum\limits_{\nu=1}^L (j_{\nu}-1) q^{\nu-1}.
\]
Assuming that for the rank-${\bf r}$-TT approximation of the quantics image $\textsf{\textbf{Y}}$ we have 
$r_k \leq r$, $k=1,\ldots ,L$, 
then the complexity of this tensor representation is reduced to the logarithmic scale
$$
q r^2 \log_q M \ll M. 
$$ 

The computational gain of the QTT approximation is justified by the 
perfect rank decomposition proven in \cite{khor-qtt-2011} for 
a wide class of function-related tensors obtained by sampling the corresponding functions
over a uniform or properly refined grid. In particular, this class of functions includes
complex exponentials, trigonometric functions, polynomials 
and Chebyshev polynomials, wavelet basis functions
(see also \cite{gras-tenz-2010,osel-constr-2013,VeBoKh:Ewald:14} for further results on QTT approximation). 

The low-rank QTT approximation can be also proven for
Gaussians, as well as for the 3D Newton, Yukawa and Helmholtz kernels. 

In the following we apply the QTT approximation method to the problem of fast integration
of highly oscillating functions introduced in the introduction.

\section{Approximation procedure}\label{sec:ApproxProc}

\subsection{One-dimensional integrals}
For simplicity we introduce the general idea of the approximation at first for one dimensional integrals
of the form
\beq
\label{int_orig}
I_{\mathscr{R}}(\omega,f):=\int_{-1}^1 f(x) \cos(\omega g(x)) dx.
\eeq
The case $I_{\mathscr{I}}(\omega,f):=\int_{-1}^1 f(x) \sin(\omega g(x)) dx$ and integrals over 
arbitrary intervals $[a,b]$ will not be treated separately since the procedure is completely 
analogous (after a suitable transformation to the interval $[-1,1]$). Recall that we are interested in the computation of integrals of the form \eqref{int_orig} for different functions $f$ and different frequencies $\omega$, hence the notation $I_{\mathscr{R}}(\omega,f)$.\\
In the following we assume that $f$ is a smooth nonoscillatory function that can be well 
approximated by polynomials of degree $N$. We introduce the 
Chebyshev polynomials by
\begin{align*}
T_0(x) &= 1,\quad T_1(x) = x,\\
T_{n+1}(x) &= 2xT_n(x)-T_{n-1}(x)\quad n=1,2,\ldots
\end{align*}

and seek an approximation of $f$ of the form
\beq
\label{ChebInterp}
f_N(x) = \sum_{k=0}^N c_k T_k(x) ,
\eeq
where $f_N$ interpolates $f$ in the Chebyshev-Gauss-Lobatto points
\[
x_k = \cos\left(\frac{k\pi}{N}\right)\quad 0\leq k \leq N.
\]
In this case the coefficients $c_k$ in \eqref{ChebInterp} are given by
\[
c_k = \frac{1}{\alpha_k N}\sum_{j=0}^{2N-1}  f\left(\cos\left(\frac{j\pi}{N}\right)\right)\cos\left(\frac{kj\pi}{N}\right),\quad k=0,\ldots,N,
\]
where
\[
\alpha_0 =\alpha_N = 2, \alpha_k = 1 \quad\text{for}\quad 1\leq k\leq N-1.
\]
The coefficients $c_k$ can be computed efficiently in $O(N\log N)$ operations using fast cosine 
transform methods. Recall that this polynomial approximation converges exponentially in $N$ if $f$ is 
sufficiently smooth. This is summarized in the following proposition.
\begin{proposition}\label{THM:ExpConv}
Let $f$ be analytic in the Bernstein regularity ellipse
\[
\mathcal{E}_\rho := \left\lbrace w\in\mathbb{C}:|w-1|+|w+1|\leq \rho + \rho^{-1}  \right\rbrace,
\]
with $\rho>1$. Furthermore let $|f(x)|\leq M_0$ in $\mathcal{E}_\rho$ for some $M_0>0$. Then the 
Chebyshev interpolant $f_N$ in \eqref{ChebInterp} satisfies
\[
\| f-f_N \|_\infty \leq \frac{4M_0}{\rho-1}\rho^{-N},\quad N\in\mathbb{N}_0.
\]
\end{proposition}
\begin{proof}
See \cite{trefethen2013}.
\end{proof}

We obtain an approximation of $I_{\mathscr{R}}(\omega,f)$ by replacing $f$ by $f_N$ and 
computing $I_{\mathscr{R}}(\omega,f_N)$ (see also \cite{Dominguez1,Dominguez2,Xiang2}, where this approach is refered to as Filon-Clenshaw-Curtis quadrature rule). The corresponding error can be easily estimated in terms of 
the error of the Chebyshev interpolation of $f$. It holds
\begin{align*}
\left| I_{\mathscr{R}}(\omega,f)-I_{\mathscr{R}}(\omega,f_N)   \right| &=\left| 
\int_{-1}^1 \left(f(x)-f_N(x)\right) \cos(\omega g(x)) dx\right|\\
&\leq \| f-f_N \|_\infty \int_{-1}^1| \cos(\omega g(x))| dx\\
&\leq 2 \| f-f_N \|_\infty.
\end{align*}
Under the assumptions of Proposition \ref{THM:ExpConv} we therefore have
\beq
\label{error_bound}
\left| I_{\mathscr{R}}(\omega,f)-I_{\mathscr{R}}(\omega,f_N)   
\right|\leq  \frac{8 M_0}{\rho-1}\rho^{-N},\quad N\in\mathbb{N}_0,
\eeq
i.e., exponential convergence of $I_{\mathscr{R}}(\omega,f_N)$ to the exact value with respect to $N$. Obviously this is also true for $I_{\mathscr{I}}(\omega,f)$. Note that the asymptotic order of this method is $\mathcal{O}(\omega^2)$ since $f_N$ interpolates $f$ in the endpoints of the interval.
\begin{remark}
In the following $f_N$ could also be written in the equivalent form
\beq
\label{LagrangeForm}
f_N(x) = \sum_{k=0}^N f(x_k) L_k(x) ,
\eeq
where
\[
L_k(x) = \prod_{\nu=0,\nu\neq k}^N\frac{x-x_{\nu}}{x_k-x_{\nu}}
\]
are the corresponding Lagrange polynomials. This representation has the advantage that the Chebyshev 
coefficient do not have to be computed. While the analysis of the scheme in Section \ref{sec:QTT_Represent_I} assumes $f_N$ to be in the form 
\eqref{ChebInterp}, the numerical experiments (see Section \ref{Sec:NumExp}) indicate that the results are very similar. 
In Section \ref{Sec:MultiDim} we will make use of \eqref{LagrangeForm}.
\end{remark}

We now turn to the question how to compute  $I_{\mathscr{R}}(\omega,f_N)$ efficiently. 
Since $f_N$ is supposed to approximate $f$ accurately this task is in general not easier 
than the original problem although $f_N$ is a polynomial. However the approximation of $f$ 
with Chebyshev polynomials leads to certain prototype functions of integrals that we want to 
precompute and store in the following. We have
\begin{align}
\label{LagExp}
I_{\mathscr{R}}(\omega,f_N) &= \sum_{k=0}^N c_k \int_{-1}^1 T_k(x) \cos(\omega g(x)) dx
= \sum_{k=0}^N c_k\,I_{\mathscr{R}}(\omega,T_k).
\end{align}
Thus, the question how to evaluate  $I_{\mathscr{R}}(\omega,f_N)$ boils down to the question 
how to efficiently evaluate  $I_{\mathscr{R}}(\omega,T_k)$ for a certain range of 
frequencies $\omega \in \left[\omega_{\min} ,\omega_{\max} \right]$, 
moderate $k$ (typically $k\leq 12$) and different oscillators $g$. Our goal is to 
precompute $I_{\mathscr{R}}(\omega,T_k)$ for fixed $k$ and $g$. Thus the 
function 
\beq
\label{int_func}
I_{\mathscr{R}}(\cdot,T_k):\left[\omega_{\min} ,\omega_{\max} \right]\rightarrow \mathbb{R}
\eeq
needs to be accurately represented and stored in order to compute an approximation 
of  $I_{\mathscr{R}}(\omega,f_N)$ via \eqref{LagExp}.  Note that this function needs to be approximated in a possibly large interval $\left[\omega_{\min} ,\omega_{\max} \right]$. Thus 
standard techniques like, e.g., polynomial interpolation/approximation are typically not effective
(see Figure \ref{quadError}). Once the prototype functions $I_{\mathscr{R}}(\cdot,T_k)$ have been precomputed for different $k$, integrals of the form \eqref{int_orig} can be easily approximated for different functions $f$ and frequencies $\omega$.


\begin{figure}[th]
\centering
\includegraphics[width=1.0\textwidth]{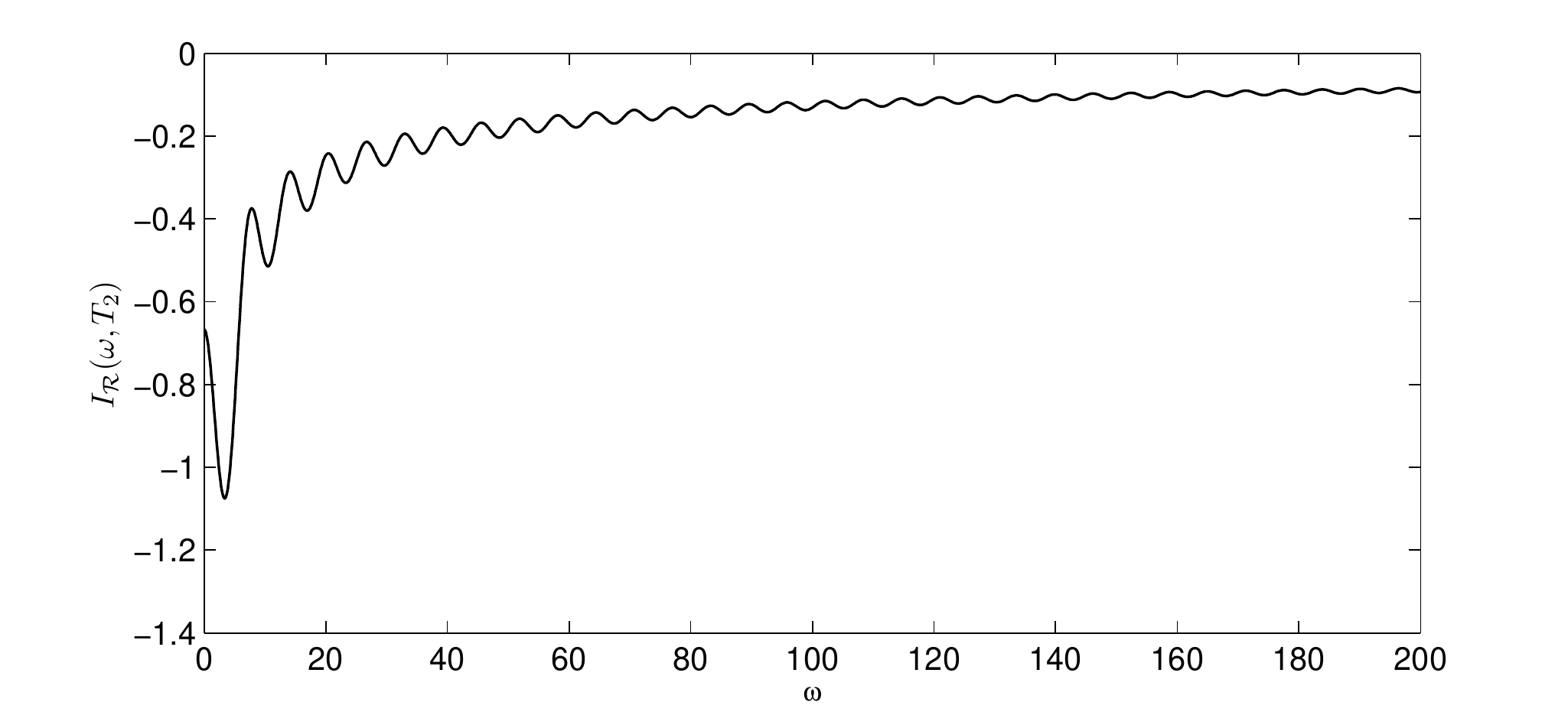}\caption{Plot of $I_{\mathscr{R}}(\omega,T_2)$ for $g(x) = x^2$.}%
\label{quadError}%
\end{figure}

In the present paper we will represent functions of the form \eqref{int_func} on the interval 
$\left[\omega_{\min} ,\omega_{\max} \right]$ pointwise on a very fine grid (see Remark \ref{Rem:GridDist} for a precise statement) via
low rank QTT tensor representations that were introduced in the preceding section. 
The straightforward  way to obtain such a representation is to evaluate $I_{\mathscr{R}}(\cdot ,T_k)$ 
at every point of the grid, to reshape the resulting vector to its quantics image and to approximate 
the resulting tensor in the TT-format. Since we seek to approximate $I_{\mathscr{R}}(\cdot,T_k)$ at 
grids that can easily exceed $2^{40}$ points this strategy is prohibitively expensive. Instead the 
final QTT tensor can be set up directly without computing the function at every point of the grid. 
This is achieved using a TT/QTT cross approximation scheme conceptually introduced \cite{Osel10}. 

It allows the computation 
of the QTT tensor using only a low number of the original tensor elements, i.e. by evaluating 
$I_{\mathscr{R}}(\cdot,T_k)$ only at a few $\omega\in \left[\omega_{\min} ,\omega_{\max} \right]$. 
More precisely, the rank-$r$ QTT-cross approximation of a $2^L$ tensor calls only $O(Lr^2)$  
entries of the original tensor. Notice that the required accuracy of the QTT approximation 
is achieved by the adaptive choice of the QTT rank $r$ within the QTT-cross approximation scheme.
The required computation of $I_{\mathscr{R}}(\cdot,T_k)$ at these 
special points is the most expensive part of the precomputation step since for fixed $\omega$ 
this is itself an oscillatory integral. Since the overall scheme to approximate 
\eqref{int_orig} is supposed to work in a black-box fashion (and we do not want to use specific 
knowledge about $g$) we suggest to compute $I_{\mathscr{R}}(\omega,T_k)$ for fixed $\omega$
(within the cross approximation scheme) by standard techniques like composite Gauss-Legendre quadrature. 
Depending on $\omega$ this certainly requires a high number of subintervals/quadrature points to achieve 
accurate results but since this has to be done only once in the precomputation step and due to 
its generality we think that this is a suitable strategy. In Section \ref{Sec:NumExp} we show that 
the time to precompute the QTT tensor is indeed very moderate in practice.

Once the functions $I_{\mathscr{R}}(\omega,T_k), \omega\in\left[\omega_{\min} ,\omega_{\max} \right]$ 
have been precomputed for $0\leq k\leq N$ and stored in the QTT format we obtain an approximation of  
$I_{\mathscr{R}}(\omega,f_N)$ (and therefore $I_{\mathscr{R}}(\omega,f)$) for a specific $\omega$ by 
evaluating the corresponding entry of the $(N+1)$ different QTT-tensors and combining them according 
to \eqref{LagExp} (see Algorithm \ref{Alg1}). Since the entry-wise evaluation of a rank-$r$ 
TT-tensor requires $O( L r^2)$ 
operations, the cost to obtain an approximation 
of $I_{\mathscr{R}}(\omega,f_N)$ from the precomputed tensors sums up to $O((N+1) L r^2)$ 
operations which is independent of $\omega$. 

\begin{remark}
\label{Rem:GridDist}
With the strategy above the function $I_{\mathscr{R}}(\cdot,T_k)$ is only represented at discrete grid points. 
In order to evaluate $I_{\mathscr{R}}(\omega,T_k)$ at an arbitrary point $\omega_0$ in 
$\left[\omega_{\min} ,\omega_{\max} \right]$, $\omega_0$ first has to be rounded to the nearest 
grid point. This leads to an additional error in 
the overall approximation of $I_{\mathscr{R}}(\omega_0,f)$ which can be easily estimated. We denote by $\tilde{\omega}_0$ the grid point that is closest to $\omega_0$. A Taylor expansion around $\tilde{\omega}_0$ shows that
\[
\left| I_{\mathscr{R}}(\tilde{\omega}_0,f)-I_{\mathscr{R}}(\omega_0,f) \right|\leq 2\left(\operatorname{e}^{|\tilde{\omega}_0-\omega_0|}-1\right).
\]
If the distance between two sampling points is $h$ the error due to rounding on this grid is therefore at most $2\left(\operatorname{e}^{h/2}-1\right)$. We therefore need $h<2\ln(\varepsilon_r/2+1)$ to assure that the error due to rounding does not exceed $\varepsilon_r$. In practice $2^{30}-2^{40}$ grid points are typically sufficient to keep the error due to rounding negligible. It becomes evident in Section \ref{Sec:NumExp} that the 
QTT approximation is well suited for such high dimensional tensors and that the ranks remain bounded.

The TT/QTT cross approximation that is used to compute the required QTT tensors is another source of errors. 
Also here we choose the approximation accuracy very high such that \eqref{error_bound} 
remains the dominant error bound.
\end{remark}

\begin{algorithm}
[H]%
\caption{Approximation of $I_{\mathscr{R}}(\omega_0,f)$}
\begin{algorithmic}
\REQUIRE $\bullet$ Precomputed QTT tensors $Q_k, 0\leq k\leq N$ that represent 
$I_{\mathscr{R}}(\omega,T_k), \omega\in\left[ \omega_{\min} ,\omega_{\max} \right]$ on a 
regular grid with $2^L$ points.\\
$\bullet$ Function $f$.\\
 $\bullet$ Value $ \omega_0\in\left[ \omega_{\min} ,\omega_{\max} \right]$. \vspace{1mm}
\hrule\medskip
\STATE Set $n\leftarrow 2^L$ and $h\leftarrow (\omega_{\max}-\omega_{\min})/(n-1)$.
\STATE Get coefficients $c_k$ of the interpolation of $f$ as in \eqref{ChebInterp}.
\STATE Set $\tilde{\omega}_0 = round(\omega_0/h-\omega_{\min}/h)$. \COMMENT {\emph{Round to the closest integer}}
\STATE Convert $\tilde{\omega_0}$ into binary representation  $\tilde{\omega}_{0,B}$. 
\COMMENT {\emph{This is the position of $\tilde{\omega}_0$ in the tensors}}
\RETURN $\sum_{k=0}^N c_k Q_k(\tilde{\omega}_{0,B})$.
\end{algorithmic}
\label{Alg1}
\end{algorithm}

\begin{remark}
\label{Rem:CrossApprox}
The availability of a cross approximation scheme is cruical for this method since otherwise the QTT 
tensors could not be computed for a large number of grid points. The main ingredient for 
its efficiency is the existence of the accurate low-rank QTT tensor approximation. As we will see in Section 
\ref{sec:QTT_Represent_I} the low rank is mainly due to the smoothness of 
$I_{\mathscr{R}}(\cdot,T_k)$. While this is always true in theory, special care has to be 
taken in practice if $I_{\mathscr{R}}(\omega,T_k)\equiv 0$, which happens if $k$ is 
odd and $g(x)$ is an even or odd function. If the cross approximation algorithm is 
applied in this case and $I_{\mathscr{R}}(\omega,T_k)$ is not evaluated exactly, it will 
try to compress a very noisy (quadrature) error function which will in general not be of low rank. 
These cases therefore have to be treated manually. The same holds for $I_{\mathscr{I}}(\omega,T_k)$.
\end{remark}

\subsection{Multi-dimensional integrals}
\label{Sec:MultiDim}

The ideas of the preceding subsection can be extended in a straightforward way to multi-dimensional 
integrals. We consider integrals of the form
\beq
\label{multiInt}
I_{\mathscr{R}}(\omega,f):= \int_{[-1,1]^d} f\left( y\right) \cos(\omega g(y)) dy ,
\eeq
where $f:[-1,1]^d\rightarrow \R$ is a smooth function and $g:[-1,1]^d\rightarrow \R$. 
We approximate $f$ by a $d$-dimensional interpolation function 
\[
f_N(y) = \sum_{j_1=0}^N\cdots\sum_{j_d=0}^N f\left( x_{j_1},\ldots ,x_{j_d} \right)
L_{j_1}(y_1)\cdots L_{j_d}(y_d),
\]
were $x_j, 0\leq j\leq N$ are again the Chebyshev points and $L_k$ are the Lagrange polynomials.
For a class of analytic function the $\varepsilon$-approximation is achieved with 
$N=|\log \varepsilon|$. Replacing $f$ by $f_N$ leads to
\beq
\label{multiApprox}
I_{\mathscr{R}}(\omega,f_N) =  \sum_{j_1=0}^N\cdots\sum_{j_d=0}^N f
\left( x_{j_1},\ldots ,x_{j_d} \right)   
\int_{[-1,1]^d}\underbrace{L_{j_1}(y_1)\cdots L_{j_d}(y_d)}_{=:L_{j_1,\ldots ,j_d}(y)}  
\cos(\omega g(y)) dy.
\eeq
Thus, by precomputing and storing $I_{\mathscr{R}}(\omega,L_{j_1,\ldots ,j_d}(y))$ for 
$\omega\in [\omega_{\min},\omega_{\max}]$ and every 
$(j_1,\ldots ,j_d)\in\left\lbrace 0,\ldots ,N\right\rbrace^d$ 
in the QTT format as described before, an approximation of $I_{\mathscr{R}}(\omega,f)$ 
for a specific $\omega$ can be obtained by evaluating the corresponding entries in the QTT 
tensors and combining them according to \eqref{multiApprox}.

If the function $g(y)$ allows certain low-rank separable representation, for example,
$g(y)=y_1^2 + y_2^2 + y_3^2$, then the representation (\ref{multiApprox}) can be presented
in the rank-$N$ Tucker type tensor format, where the $d$-dimensional integrals 
in the right-hand site of (\ref{multiApprox}) are reduced to 1D integrations.
Notice that in the latter example the representation in complex arithmetics 
leads to lower rank parameters.

\section{QTT tensor approximation of the functions $I_{\mathscr{R}}(\cdot,T_k)$ and $I_{\mathscr{I}}(\cdot,T_k)$}
\label{sec:QTT_Represent_I}

In this section we show that the functions  $I_{\mathscr{R}}(\cdot,T_k)$ and $I_{\mathscr{I}}(\cdot,T_k)$, 
sampled on a uniform grid, can be efficiently represented in the QTT format by deriving explicit ranks 
bounds of these approximations.

A first important observation is that $I_{\mathscr{R}}(\omega,T_k)$ and $I_{\mathscr{I}}(\omega,T_k)$ 
are very smooth functions with respect to $\omega$, independent of the smoothness of $g$.

\begin{lemma}\label{lem:entire}
The functions $I_{\mathscr{R}}(\cdot,T_k):\Co\rightarrow \Co$ and 
$I_{\mathscr{I}}(\cdot,T_k):\Co\rightarrow \Co$ are entire.
\end{lemma}
\begin{proof}  We have
\begin{align*}
I_{\mathscr{R}}(\omega_{\operatorname{Re}}+i \omega_{\operatorname{Im}},T_k) &= 
\int_{-1}^1 T_k(x) \cos(\omega_{\operatorname{Re}}g(x)+i \omega_{\operatorname{Im}} g(x)) dx\\
&=\int_{-1}^1 T_k(x) \cos(\omega_{\operatorname{Re}}g(x))\cosh( \omega_{\operatorname{Im}} g(x)) dx \\
&\quad + i\int_{-1}^1 -T_k(x) \sin(\omega_{\operatorname{Re}}g(x))\sinh( \omega_{\operatorname{Im}} g(x)) dx\\
&=: u(\omega_{\operatorname{Re}}, \omega_{\operatorname{Im}}) + 
i\cdot v(\omega_{\operatorname{Re}}, \omega_{\operatorname{Im}})
\end{align*}
The real valued functions $u$ and $v$ have continuous first partial derivatives 
and satisfy the Cauchy-Riemann differential equations:
\[
\frac{\partial u}{\partial \omega_{\operatorname{Re}} }= \int_{-1}^1 -T_k(x)g(x) 
\sin(\omega_{\operatorname{Re}}g(x))\cosh( \omega_{\operatorname{Im}} g(x)) dx 
= \frac{\partial v}{\partial \omega_{\operatorname{Im}} },
\]
\[
\frac{\partial u}{\partial \omega_{\operatorname{Im}} }= 
\int_{-1}^1 T_k(x)g(x) \cos(\omega_{\operatorname{Re}}g(x))
\sinh( \omega_{\operatorname{Im}} g(x)) dx = -\frac{\partial v}{\partial 
\omega_{\operatorname{Re}} }.
\]
Thus $I_{\mathscr{R}}(\omega,T_k)$ is holomorphic and therefore entire. 
A similar reasoning applies to $I_{\mathscr{I}}(\omega,T_k)$.
\end{proof}

In some applications the so-called $sinc$-approximation method can be applied as an alternative to the polynomial
approximation. In this case the band-limitedness of the target function plays an important role.
The following lemma provides the respective analysis.
\begin{lemma}
Let $g\in C^1[-1,1]$ be invertible. Then the function $I(\cdot,f,g):\R\rightarrow \R$ is 
band-limited.
\end{lemma}
\begin{proof} Let $h$ be the inverse function of $g$ and $H$ denote the Heaviside step function. 
Then
\begin{align*}
I(\omega,f,g) &= \int_{-1}^1 f(x) \operatorname{e}^{i\omega g(x)} dx = 
\int_{g(-1)}^{g(1)} f(h(x))h^\prime(x) \operatorname{e}^{i\omega x} dx\\
&= \frac{1}{2\pi}\int_{-\infty}^{\infty} 2\pi f(h(x))h^\prime(x)H(g(1)-x)H(x-g(-1)) 
\operatorname{e}^{i\omega x} dx\\
&= \mathcal{F}^{-1}\left[ 2\pi f(h(x))h^\prime(x)H(g(1)-x)H(x-g(-1)) \right](\omega),
\end{align*}
where $\mathcal{F}^{-1}$ denotes the inverse Fourier transform. 
This shows that $\mathcal{F}(I(\cdot,f,g))$ is compactly supported.
\end{proof}

Next, we establish bounds of $I_{\mathscr{R}}(\omega,T_k)$ and $I_{\mathscr{I}}(\omega,T_k)$ in the complex domain.
\begin{lemma}\label{lem:Bound_on_I}
Let $\omega\in\Co$, $k\in\N_0$ and $|g(x)|\leq 1$ for $x\in[-1,1]$. Then
\[
\left| I_{\mathscr{R}}(\omega,T_k)\right|\leq 2 
\operatorname{e}^{|\omega_{\operatorname{Im}}|}\quad\text{and}
\quad \left| I_{\mathscr{I}}(\omega,T_k)\right|\leq 2 
\operatorname{e}^{|\omega_{\operatorname{Im}}|}.
\]
\end{lemma}
\begin{proof} It holds
\begin{align*}
\left| I(\omega,T_k)\right|&= \left| \int_{-1}^1 T_k(x) \operatorname{e}^{i\omega g(x)} dx 
\right| = \left| \int_{-1}^1 T_k(x) \operatorname{e}^{i\omega_{\operatorname{Re}} g(x)-
\omega_{\operatorname{Im}} g(x)} dx \right|\\
&\leq  \int_{-1}^1  \operatorname{e}^{-\omega_{\operatorname{Im}} g(x)}  \left|  
\operatorname{e}^{i\omega_{\operatorname{Re}} g(x)}\right| dx \\
&\leq 2 \operatorname{e}^{|\omega_{\operatorname{Im}}|}.
\end{align*}
This bound holds both for $I_{\mathscr{R}}(\omega,T_k)$ and 
$I_{\mathscr{R}}(\omega,T_k)$.
\end{proof}

The next theorem establishes explicit rank bounds of the QTT approximations of $I_{\mathscr{R}}(\omega,T_k)$ and $I_{\mathscr{I}}(\omega,T_k)$.
\begin{theorem}
\label{THM:main}
Let the function  $I_{\mathscr{R}}(\omega,T_k)$ or $I_{\mathscr{I}}(\omega,T_k)$ be 
sampled on the uniform grid $\omega_{\min}=x_0<x_1<\ldots <x_N=\omega_{\max}, x_i = \omega_{\min}+hi$  in the interval 
$[-\omega_{\min},\omega_{\max}]$ with $N=2^L$ and call the resulting vector $\underline{\nu}$. 
Let furthermore $1>\varepsilon >0$ be given.
Then there exists a QTT approximation $\nu_{\text{QTT}}$ of  $\underline{\nu}$ with 
ranks bounded by
\begin{equation}
\label{rankbound}
r\left( \nu_{\text{QTT}} \right) \leq  1+ \ln\left(\frac{8}{\operatorname{e}-1}\right)+C(\omega_{\max}-\omega_{\min})+\ln\left(\frac{1}{\varepsilon}\right) 
\end{equation}
with $C:=\frac{1}{4}\sqrt{(\operatorname{e}+\operatorname{e}^{-1})^2-4}\approx 0.59$ and accuracy
\[
|\underline{\nu}- \nu_{\text{QTT}}|\leq \varepsilon.
\]
\end{theorem}

\begin{proof}
We perform the proof only for $I_{\mathscr{R}}(\omega,T_k)$ since the case 
$I_{\mathscr{I}}(\omega,T_k)$ is analogous. We consider a polynomial approximation of $I_{\mathscr{R}}(\omega,T_k)$ as in Proposition \ref{THM:ExpConv}. We define the linear scaling
\[
\chi_{\omega_{\min},\omega_{\max}}:[-1,1]\rightarrow [\omega_{\min},\omega_{\max}], \quad \chi_{\omega_{\min},\omega_{\max}}(\omega) = \frac{\omega_{\max}-\omega_{\min}}{2}\cdot\omega+\frac{\omega_{\max}+\omega_{\min}}{2}
\]
and the transformed function
\[
I_{\mathscr{R}}^{[-1,1]}(\cdot,T_k) :[-1,1]\rightarrow\mathbb{R},\quad I_{\mathscr{R}}^{[-1,1]}(\omega,T_k) = I_{\mathscr{R}}(\chi_{\omega_{\min},\omega_{\max}}(\omega),T_k)
\]
As in Lemma \ref{lem:Bound_on_I} we can show that
\[
\left|I_{\mathscr{R}}^{[-1,1]}(\omega,T_k)\right|\leq 2\operatorname{e}^{|\omega_{\text{Im}}|\cdot\left( \frac{\omega_{\max}-\omega_{\min}}{2}\right)}
\]
for $\omega\in\mathbb{C}$. We approximate $I_{\mathscr{R}}^{[-1,1]}(\omega,T_k)$ in $[-1,1]$ with a Chebyshev interpolant $I_N^{[-1,1]}(\omega)$. Since  $I_{\mathscr{R}}^{[-1,1]}(\omega,T_k)$ is entire and due to the bound above in the complex plane, Proposition \ref{THM:ExpConv} (with $\rho=\operatorname{e}$) shows that
\[
\|I_{\mathscr{R}}^{[-1,1]}(\cdot,T_k) - I_N^{[-1,1]}\|_\infty \leq \frac{4M_0}{\operatorname{e}-1}\operatorname{e}^{-N}
\]
with
\[
M_0 = 2\operatorname{e}^{C(\omega_{\max}-\omega_{\min})},\quad C:=\frac{1}{4}\sqrt{(\operatorname{e}+\operatorname{e}^{-1})^2-4}.
\]
Thus we have to choose
\[
N\geq \ln\left(\frac{4M_0}{\operatorname{e}-1}\right)+\ln\left(\frac{1}{\varepsilon}\right)
\]
in order to assure that the approximation error satisfies $\|I_{\mathscr{R}}^{[-1,1]}(\cdot,T_k) - I_N^{[-1,1]}\|_\infty\leq \varepsilon$. Since polynomials of degree $N$ sampled on a uniform grid have QTT ranks bounded by $N+1$ the assertion follows. 
\end{proof}

Theorem \ref{THM:main} shows that the QTT ranks depend only logarithmically on the desired accuracy of the approximation. On the other hand it also suggests that the ranks depend linearly on the size of the approximation interval $[\omega_{\min},\omega_{\max}]$. Such a linear dependence could not be observed in practice. We demonstrate in Section \ref{Sec:NumExp} that the ranks stay small even if large intervals $[\omega_{\min},\omega_{\max}]$ are considered.

\section{Numerical experiments}
\label{Sec:NumExp}

In this section we present the results of the numerical experiments. All computations were performed in 
MATLAB using the TT-Toolbox 2.2  (\url{http://spring.inm.ras.ru/osel/}).\\
The main goal in this section is to show that the functions  $I_{\mathscr{R}}(\omega,T_k)$ and $I_{\mathscr{I}}(\omega,T_k)$ indeed admit a representation in the QTT format with low ranks. As discussed above low ranks are cruical for the cross approximation as well as the efficiency for Algorithm \ref{Alg1}.

\subsection{One-dimensional integrals}

At first we verify the exponential convergence (with respect to $N$) of the scheme that is predicted 
in \eqref{error_bound}. For this we set $g(x)=x^2$ and $g(x) = \sin(x+1)$ and precompute the QTT 
tensors representing the functions $I_{\mathscr{R}}(\omega,T_k)$ and $I_{\mathscr{I}}(\omega,T_k)$  
with $\omega\in [0,1000]$ at $2^{63}$ points. The relative error 
$err_{\omega}(N) = \left| I(\omega,f)- I(\omega,f_N) \right|/ \left| I(\omega,f)\right|$ for 
different values of $\omega$ and functions $f$ is illustrated in Figure \ref{FIG:Convergence}. 
The ``exact'' value $ I(\omega,f)$ was computed using a high order composite Gauss-Legendre quadrature rule. 
It becomes evident that the quadrature error decays indeed exponentially and that already low numbers 
of $N$ lead to accurate approximations.

\begin{figure}[!h]
\centering
\subfigure[Relative error of the approximation of  $\int_{-1}^1 \cos(x)\operatorname{e}^{i\omega x^2}$]{
\centering
\includegraphics[width=0.47\textwidth]{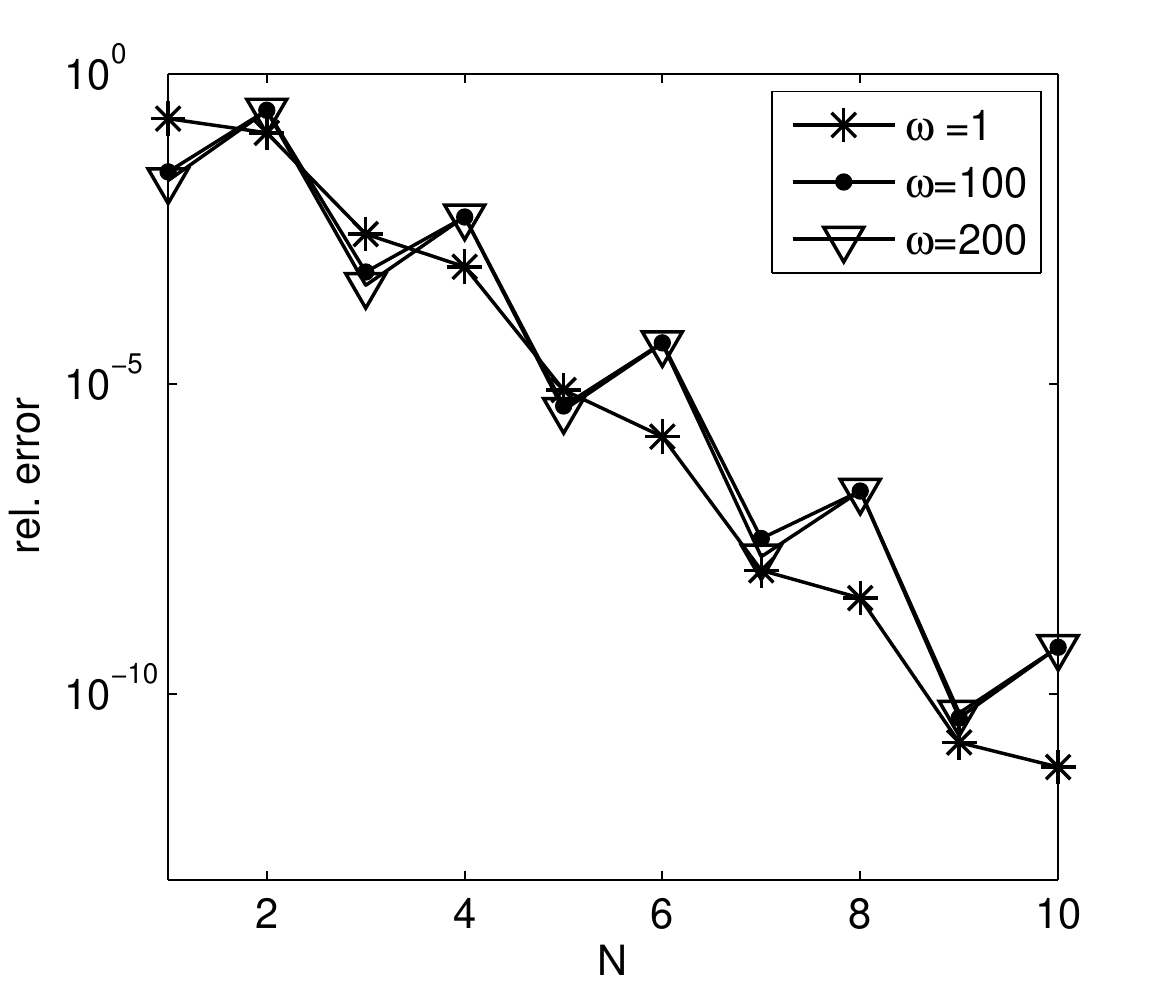}
} \hfil
\subfigure[Relative error of the approximation of  $\int_{-1}^1 \cos(x+1)\operatorname{e}^{i\omega \sin(x+1)}$]{
\centering
\includegraphics[width=0.47\textwidth]{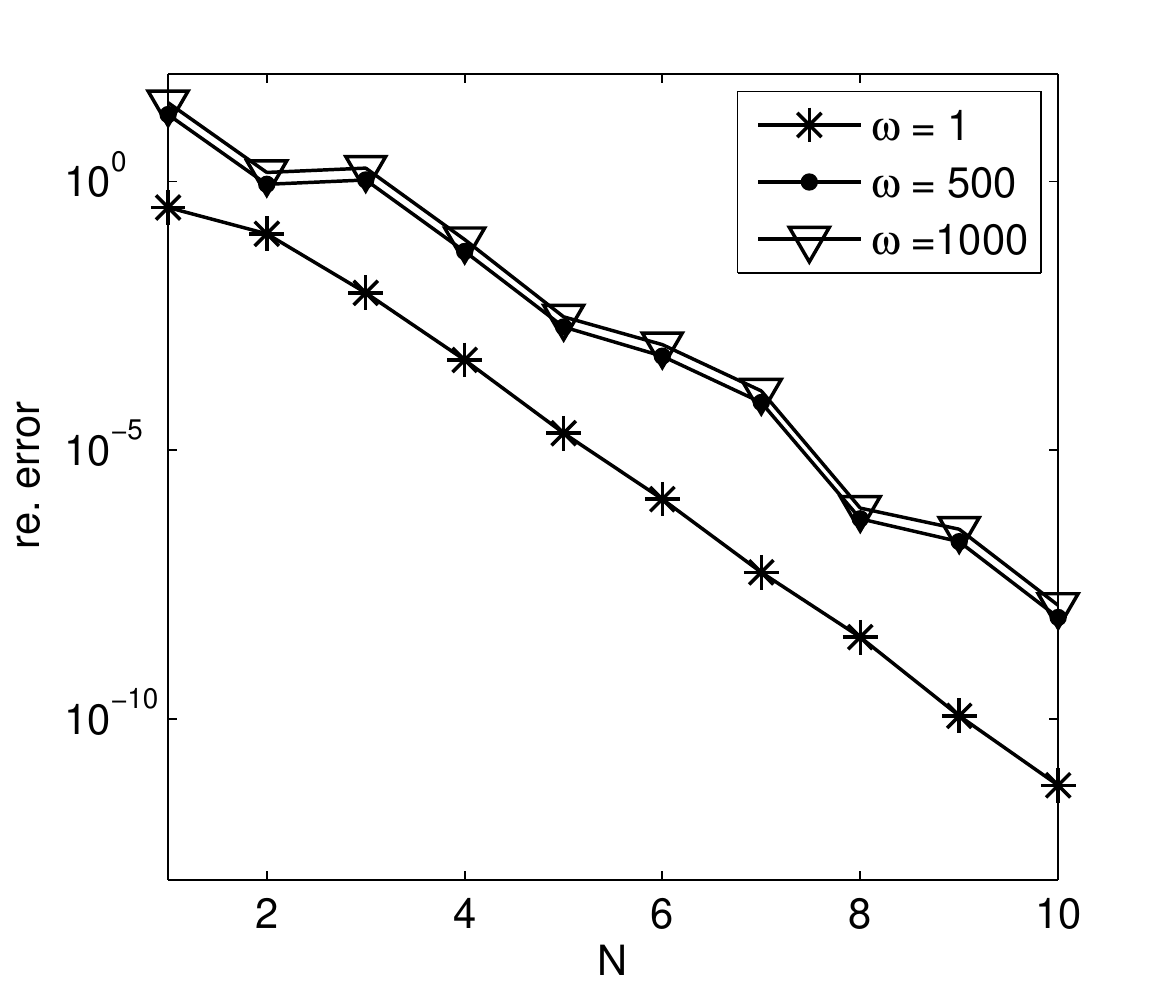}
} \caption{Quadrature error with respect to the polynomial degree 
$N$ for different integrals.}%
\label{FIG:Convergence}%
\end{figure}

Table \ref{table1} and Table \ref{table2} show the effective ranks of the QTT approximations of 
$I_{\mathscr{R}}(\omega,T_k)$ and $I_{\mathscr{I}}(\omega,T_k)$ in various situations. 
It becomes evident that the choice of $g$ has only a minor influence on the corresponding QTT ranks. 
Even non-smooth functions, functions with stationary points and functions with unbounded first 
derivative are unproblematic. Another observation is that the influence of $k$ 
(degree of the Chebyshev polynomial) on the QTT ranks is very moderate even though the 
functions $I_{\mathscr{R}}(\omega,T_k)$ and $I_{\mathscr{I}}(\omega,T_k)$ become more oscillatory 
with increasing $k$. This is supported by the theory with states rank bounds that are 
independent of $k$.\vspace{\baselineskip}

\begin{table}[h]
\centering
\begin{tabular}{cc|cccc}
$M$ & $[\omega_{\min},\omega_{\max}]$ & $\begin{matrix}g(x)=x,\\ k=2 \end{matrix}$  & 
$\begin{matrix}g(x)=x,\\ k=10 \end{matrix}$  &  
$\begin{matrix}g(x)=\frac{1}{2}x^2+\frac{1}{4}x,\\ k=2 \end{matrix}$  &  
$\begin{matrix}g(x)=\frac{1}{2}x^2+\frac{1}{4}x,\\ k=10 \end{matrix}$  \\ 
\hline
$2^{40}$ & $[0,100]$  & $4.6$  & $4.9$  & $4.5$ & $4.7$ \\ 
$2^{50}$ & $[0,100]$ & $4.2$ & $4.4$  & $4.2$ & $4.2$ \\ 
$2^{60}$ & $[0,100]$ & $3.8$  & $4.1$  & $3.8$ & $3.9$ \\ \hline
$2^{43}$ & $[0,1000]$ & $5.8$ & $6.1$ & $6.4$ & $6.6$  \\ 
$2^{53}$ & $[0,1000]$ & $5.2$  & $5.5$ & $5.8$ & $6.0$ \\ 
$2^{63}$ & $[0,1000]$ & $4.9$ & $5.0$  & $5.4$ & $5.4$ \\ \hline
$2^{43}$ & $[0,2000]$ & $6.2$ & $6.5$ & $7.2$ & $7.4$ \\ 
$2^{53}$ & $[0,2000]$ & $5.6$ & $5.9$ & $6.5$ & $6.7$ \\ 
$2^{63}$ & $[0,2000]$ & $5.2$ & $5.4$ & $6.0$ & $6.1$ \\ 
\end{tabular}
\caption{Effective QTT-ranks of $M$-vectors related to the function 
$I_{\mathscr{R}}(\omega,T_k)$ sampled on a uniform grid in the interval 
$[\omega_{\min},\omega_{\max}]$.} 
\label{table1}
\end{table}

\begin{table}[h]
\centering
\begin{tabular}{cc|ccc}
$M$ & $[\omega_{\min},\omega_{\max}]$ & $g(x)=\cos\left(x+\frac{1}{4} \right)$  &   
$g(x)=\operatorname{e}^x$  &  $g(x)=\sin(x)^2\sqrt{x+1}$  \\ 
\hline
$2^{42}$ & $[0,500]$  &  $6.2$ &  7.2 &  6.0 \\ 
$2^{52}$ & $[0,500]$ & $5.6$ &  6.5 &  5.5 \\ 
$2^{62}$ & $[0,500]$ & $5.2$  & 6.0  & 5.0  \\ \hline
$2^{42}$ & $[500,700]$ & $4.8$ & 5.1 &   4.5 \\ 
$2^{52}$ & $[500,700]$ &  $4.4$ & 4.6 &  4.0 \\ 
$2^{62}$ & $[500,700]$ & $4.1$ & 4.3  &  3.7 \\ 
\end{tabular}
\caption{Effective QTT-ranks of $M$-vectors related to the function 
$I_{\mathscr{I}}(\omega,T_5)$ sampled on a uniform grid in the interval 
$[\omega_{\min},\omega_{\max}]$.}
\label{table2}
\end{table}

If the approximation of $f$ is written in the form \eqref{LagrangeForm} then the functions 
$I_{\mathscr{R}}(\omega,L_k)$ and $I_{\mathscr{I}}(\omega,L_k)$, where $L_k$ is the $k$-th 
Lagrange basis polynomial, have to be precomputed. Table \ref{table3} shows the effective 
ranks of the QTT approximations of these functions. It can be seen that also in this case the 
ranks are low for different oscillators $g$, different $k$ and intervals $[\omega_{\min},\omega_{\max}]$. 
Using the Lagrange form of $f$ can be advantageous if  $I(\omega,f)$ has to be approximated for 
many different functions $f$ since the Chebyshev coefficients do not have to be computed. 
On the other hand the form \eqref{ChebInterp} can be favorable if the oscillator $g$ is an even or 
odd function since in this case the functions $I_{\mathscr{R}}(\omega,T_k)$ and $I_{\mathscr{I}}(\omega,T_k)$ 
are identical to zero for certain $k$ and therefore have neither to be precomputed 
nor evaluated (see Remark \ref{Rem:CrossApprox}).\vspace{\baselineskip}

\begin{table}[htb]
\centering
\begin{tabular}{cc|cccc}
$M$ & $[\omega_{\min},\omega_{\max}]$ & $\begin{matrix}g(x)=x,\\ k=1 \end{matrix}$  & 
$\begin{matrix}g(x)=x,\\ k=5 \end{matrix}$  &  
$\begin{matrix}g(x)=\frac{1}{2}x^2+\frac{1}{4}x,\\ k=1 \end{matrix}$  &  
$\begin{matrix}g(x)=\frac{1}{2}x^2+\frac{1}{4}x,\\ k=5 \end{matrix}$  \\ 
\hline
$2^{43}$ & $[0,1000]$ & $6.1$ & $5.8$ & $6.6$ & $6.4$  \\ 
$2^{53}$ & $[0,1000]$ & $5.5$  & $5.3$ & $6.0$ & $5.8$ \\ 
$2^{63}$ & $[0,1000]$ & $5.0$ & $4.9$  & $5.5$ & $5.3$ \\
\end{tabular}
\caption{Effective QTT-ranks of $M$-vectors related to the function 
$I_{\mathscr{R}}(\omega,L_k)$ sampled on a uniform grid in the interval 
$[\omega_{\min},\omega_{\max}]$.} 
\label{table3}
\end{table}

In Table \ref{table4} the computational time is illustrated that is needed to precompute  
$I_{\mathscr{R}}(\cdot ,T_k)$ via the QTT cross approximation algorithm. These timings mainly 
depend on the method that is used to compute $I_{\mathscr{R}}(\omega_0,T_k)$ at different 
points $\omega_0\in [\omega_{\min},\omega_{\max}]$ within this algorithm. As mentioned above 
we use a standard Gauss-Legendre quadrature rule to approximate these integrals. 
More precisely we divide the integration domain $[-1,1]$ into  $\omega_{\max}$ subintervals and 
use $8$ quadrature points in each subinterval. Although this strategy has a suboptimal complexity, 
it is very general, easy to implement and leads to very accurate approximations of $I_{\mathscr{R}}(\omega_0,T_k)$.  
Since the cross approximation algorithm requires the computation of $I_{\mathscr{R}}(\omega_0,T_k)$ 
only at few grid points the computing times remain very moderate for different intervals, 
grid sizes and oscillators. Recall that in order to obtain approximations of $I_{\mathscr{R}}(\omega,f)$ 
the functions $I_{\mathscr{R}}(\omega,T_k)$ have to be precomputed for $0\leq k\leq N$. 
Since these tasks are independent of each other, they can be easily performed in parallel.\vspace{\baselineskip}

\begin{table}[ht]
\centering
\begin{tabular}{c|ccc|ccc|ccc}
 $M$ & $2^{43}$ & $2^{53}$ & $2^{63}$ & $2^{43}$ & $2^{53}$ & $2^{63}$ & $2^{43}$ & $2^{53}$ & $2^{63}$ \\ 
  $[\omega_{\min},\omega_{\max}]$ & \multicolumn{3}{c|}{ $[0,500]$} & \multicolumn{3}{c|}{ $[500,700]$} & \multicolumn{3}{c}{ $[0,1000]$} \\ 
\hline
  $\begin{matrix}g(x)=x,\\ k=2 \end{matrix}$ & $12s$ & $15s$  & $17s$  & $12s$ & $17s$ & $19s$ & $20s$ & $24s$ & $25s$  \\ \hline
  $\begin{matrix}g(x)=\frac{1}{2}x^2+\frac{1}{4}x,\\ k=3 \end{matrix}$ &  $12s$ & $13s$  & $15s$  & $14s$ & $17s$ & $20s$ & $19s$ & $21s$ & $25s$  \\ 
\end{tabular}
\caption{Computing times for the precomputation of $I_{\mathscr{R}}(\omega,T_k)$ in seconds on an Intel Core i7-2600K processor.} 
\label{table4}
\end{table}

\emph{Fourier integrals}\vspace{\baselineskip}

Let $f:\R\rightarrow\Co$ be an integrable function with $\operatorname{supp}f = [a,b]$. 
We are interested in computing the Fourier transform
\begin{align}
\label{fourier}
\hat{f}(\omega) &= \int_{a}^{b}f(x)\operatorname{e}^{- \operatorname{i} x\omega} dx
\end{align}
at different random points $\omega\in [\omega_{\min},\omega_{\max}]$. We define the affine 
scaling function $\chi_{a,b}(x) := \frac{b-a}{2}x+\frac{b+a}{2}$ and obtain
\begin{align*}
\hat{f}(\omega) &= \frac{b-a}{2} \int_{-1}^{1} f\left(\chi_{a,b}(x)\right)
\operatorname{e}^{- \operatorname{i}\cdot \chi_{a,b}(x)\cdot\omega} dx \\
&= \frac{b-a}{2} \operatorname{e}^{-\operatorname{i}\cdot \frac{b+a}{2}\cdot\omega} 
\int_{-1}^{1} f\left(\chi_{a,b}(x)\right)\operatorname{e}^{- \operatorname{i} x\tilde{\omega}} dx.
\end{align*}
with  $\tilde{\omega}= \frac{b-a}{2}\omega$. The integration domain and the oscillator of the last 
integral are independent of $a$ and $b$. Thus, in order to approximate integrals of the form 
\eqref{fourier} for different intervals $[a,b]$, only the functions $I_{\mathscr{R}}(\cdot,T_k)$ 
and $I_{\mathscr{I}}(\cdot,T_k)$  with $g(x) = -x$ have to be precomputed in 
a sufficiently large interval. Note the ranks of the corresponding QTT tensors that can be observed numerically are very similar to the results in Table \ref{table3}.  \vspace{\baselineskip}


\emph{Exotic oscillators}\vspace{\baselineskip}

Now, we consider oscillators which are not of the form $h_\omega (x) = 
\operatorname{e}^{\operatorname{i}\omega g(x)}$. Levin-type methods were introduced for the case of 
the Bessel oscillator $h_\omega (x) = J_\nu(\omega x)$ in \cite{Xiang}, 
Filon-type methods were considered in \cite{IserlesLevin} for oscillators of the 
form $h_\omega (x) = v(\sin(\omega \theta(x)) )$. For most other types of oscillators such 
methods are not available so far. In Table \ref{table_exotic} we show that our method can also 
be applied in these (and other) cases in the same way as before. Low ranks of the QTT approximations of the functions
\[
I_{h_\omega,k}(\omega):= \int_{-1}^1 T_k(x) h_\omega(x) dx, \quad 1\leq k\leq N
\] 
can be observed in all tested cases.

\begin{table}[h]
\centering
\begin{tabular}{cc|cccc}
$M$ & $[\omega_{\min},\omega_{\max}]$ & $\begin{matrix}h_\omega(x)=\\ J_{11}(\omega x) \end{matrix}$  & 
$\begin{matrix}h_\omega(x)=\\  J_{8}^2(\omega x^2) \end{matrix}$  &  
$\begin{matrix}h_\omega(x)=\\ \cos(\sin(\omega x)+1 ) \end{matrix}$  &  
$\begin{matrix}h_\omega(x)=\\  \Gamma(0.5\cdot\sin(\omega x)+2 )\end{matrix}$  \\ 
\hline
$2^{40}$ & $[0,500]$  & $5.4$  & $5.7$  & $7.2$ & $7.3$ \\ 
$2^{50}$ & $[0,500]$ & $4.9$ & $5.2$  & $6.5$ & $6.5$ \\ 
$2^{60}$ & $[0,500]$ & $4.5$  & $4.7$  & $6.0$ & $5.9$ \\
\end{tabular}
\caption{Effective QTT-ranks of $M$-vectors related to the function 
$I_{h_\omega,5}(\omega)$ sampled on a uniform grid in the interval 
$[\omega_{\min},\omega_{\max}]$.} 
\label{table_exotic}
\end{table}

\subsection{Multi-dimensional integrals}

In this paragraph we consider functions of the form
\beq
\label{multiDfun}
I_{\mathscr{R}}(\omega,L_{j_1,\ldots ,j_d}) =   
\int_{[-1,1]^d} L_{j_1}(y_1)\cdots L_{j_d}(y_d)  \cos(\omega g(y)) dy
\eeq
whose precomputation is necessary when \eqref{multiApprox} is used to approximate \eqref{multiInt}. 
Table \ref{table_multi} and \ref{table_multi2} show the effective QTT ranks of the corresponding tensors for different 
oscillators $g$ in 2 and 3 dimensions. As before we can observe that the ranks are small 
in all tested situations. The computation of the integrals within the cross approximation scheme 
was performed using a tensorized version of the Gauss-Legendre quadrature described before.


\begin{table}[h]
\centering
\begin{tabular}{cc|cccc}
$M$ & $[\omega_{\min},\omega_{\max}]$ & $\begin{matrix}g(x)=\\ x_1+x_2 \end{matrix}$  &  
$\begin{matrix}g(x)=\\ \sin(x_1)/\sqrt{x_1x_2+3} \end{matrix}$   \\ 
\hline
$2^{30}$ & $[0,100]$  & $5.2$    & $4.4$ \\ 
$2^{40}$ & $[0,100]$ & $4.6$   & $3.9$  \\ 
$2^{50}$ & $[0,100]$ & $4.2$   & $3.5$  
\end{tabular}
\caption{Effective QTT-ranks of $M$-vectors related to the function 
$I_{\mathscr{R}}(\omega,L_{2,5})$ for $d=2$  sampled on a uniform grid in the interval 
$[\omega_{\min},\omega_{\max}]$.} 
\label{table_multi}
\end{table}

\begin{table}[h]
\centering
\begin{tabular}{cc|cccc}
$M$ & $[\omega_{\min},\omega_{\max}]$ & 
$\begin{matrix}g(x)=\\ x_1+x_2+x_3 \end{matrix}$    &  
$\begin{matrix}g(x)=\\ \sin(x_1x_3)/\sqrt{x_1x_2+3} \end{matrix}$  \\ 
\hline
$2^{30}$ & $[0,50]$    & $6.1$   & $4.3$ \\ 
$2^{40}$ & $[0,50]$ &    $5.5$   & $3.9$ \\ 
$2^{50}$ & $[0,50]$   & $5.3$   & $3.7$ 
\end{tabular}
\caption{Effective QTT-ranks of $M$-vectors related to the function 
 $I_{\mathscr{R}}(\omega,L_{2,5,3})$ for $d=3$ sampled on a uniform grid in the interval 
$[\omega_{\min},\omega_{\max}]$.} 
\label{table_multi2}
\end{table}

\section{Conclusion}

We described a new approach for the efficient approximation of 
highly oscillatory weighted integrals.
The main idea of our approach is to compute a priory and then represent
in low-parametric tensor formats   certain $\omega$-dependent 
prototype functions (which by itself consist of oscillatory integrals), 
whose evaluation lead in a straightforward way to approximations of
the target integral. 
The QTT approximation method for long functional $m$-vectors
allows the accurate approximation and efficient $\log m$-storage of these functions in 
the wide range of grid and frequency parameters.
Numerical examples illustrate the efficiency of the QTT-based numerical integration
scheme on many nontrivial examples in one and several spatial dimensions.
This demonstrates the promising features of the method for further applications to 
the general class of highly oscillating integrals and for the solution of ODEs and
PDEs with oscillating or/and quasi-periodic coefficients arising in computational 
physics and chemistry as well as in homogenization techniques.

\vspace{\baselineskip}

\textbf{Acknowledgement.} A large  part of this research was conducted during 
a stay of the second author at 
the Max Planck Institute for Mathematics in the Sciences in Leipzig. 
The financial support is greatly acknowledged.

\bibliographystyle{abbrv}
\bibliography{mybib}

\end{document}